\documentclass[11pt,reqno]{amsart}
\usepackage{amssymb,enumitem,setspace,tikz,xcolor,mathrsfs,listings,multicol,mathrsfs,rotating}
\usepackage{rotating}
\usepackage[vcentermath]{youngtab}
\usepackage{hyperref}
\usetikzlibrary{arrows,matrix}
\usepackage[margin=1.125in]{geometry}
\tikzset{tab/.style={matrix of math nodes,column sep=-.35, row sep=-.35,text height=7pt,text width=7pt,align=center,inner sep=2,font=\footnotesize}}

\newcommand{\g}{\mathfrak{g}}
\newcommand{\HH}{\mathcal{H}}

\newcommand{\im}{\mathrm{im}}
\newcommand{\inner}[2]{\left\langle #1, #2 \right\rangle}
\newcommand{\iso}{\cong}

\renewcommand{\mid}{:}
\newcommand{\NN}{\mathbf{N}}
\newcommand{\Oint}{\mathcal{O}_{\mathrm{int}}}
\newcommand{\one}{\boldsymbol{1}}

\newcommand{\QQ}{\mathbf{Q}}
\newcommand{\RC}{\operatorname{RC}} 
\newcommand{\re}{\mathrm{re}}

\newcommand{\wt}{\mathrm{wt}}
\newcommand{\zero}{\boldsymbol{0}}
\newcommand{\ZZ}{\mathbf{Z}}
\newcommand{\tep}{\widetilde{\varepsilon}} 

\definecolor{darkred}{rgb}{0.7,0,0} 
\newcommand{\defn}[1]{{\color{darkred}\emph{#1}}} 

\hypersetup{
 colorlinks=True,
 citecolor=green!50!black,
 urlcolor=darkred,
 linkcolor=blue,
}

\usepackage{listings}
\lstdefinelanguage{Sage}[]{Python}
{morekeywords={False,True},sensitive=true}
\usepackage{xparse}

\makeatletter
\protected\def\specialmergetwolists{%
  \begingroup
  \@ifstar{\def\cnta{1}\@specialmergetwolists}
    {\def\cnta{0}\@specialmergetwolists}%
}
\def\@specialmergetwolists#1#2#3#4{%
  \def\tempa##1##2{%
    \edef##2{%
      \ifnum\cnta=\@ne\else\expandafter\@firstoftwo\fi
      \unexpanded\expandafter{##1}%
    }%
  }%
  \tempa{#2}\tempb\tempa{#3}\tempa
  \def\cnta{0}\def#4{}%
  \foreach \x in \tempb{%
    \xdef\cnta{\the\numexpr\cnta+1}%
    \gdef\cntb{0}%
    \foreach \y in \tempa{%
      \xdef\cntb{\the\numexpr\cntb+1}%
      \ifnum\cntb=\cnta\relax
        \xdef#4{#4\ifx#4\empty\else,\fi\x#1\y}%
        \breakforeach
      \fi
    }%
  }%
  \endgroup
}
\makeatother

\usepackage{xparse}
\DeclareDocumentCommand\rpp{ m m g }{
	\foreach \x [count=\s from 1] in {#1}{
	        {\ifnum\s=1
	                \draw (0,-\s)--(\x,-\s);
	                \fi}
	   \draw (0,-\s-1) to (\x,-\s-1);
	   \foreach \y in {0, ..., \x} {\draw (\y,-\s)--(\y,-\s-1);}
	}
	\specialmergetwolists{/}{#1}{#2}\ziplist
	\foreach \x/\y [count=\yi from 1] in \ziplist{
	    \node[anchor=west,font=\scriptsize] at (\x,-\yi - .5) {$\y$};
	}
	\IfValueT {#3}
	{\foreach \z [count=\zi from 1] in {#3} {\node[anchor=east,font=\scriptsize] at (0,-\zi - .5) {$\z$};}}
	{}
}

\theoremstyle{plain}
\newtheorem{thm}{Theorem}[section]
\newtheorem{lemma}[thm]{Lemma}

\newtheorem{prop}[thm]{Proposition}
\newtheorem{cor}[thm]{Corollary}
\theoremstyle{definition}
\newtheorem{dfn}[thm]{Definition}
\newtheorem{ex}[thm]{Example}
\newtheorem{remark}[thm]{Remark}

\numberwithin{equation}{section}
\numberwithin{figure}{section}
\numberwithin{table}{section}
\usepackage[colorinlistoftodos]{todonotes}

\begin{document}
\title[RCs and $\ast$ for generalized Kac--Moody algebras]{Rigged configurations and the $\ast$-involution for generalized Kac--Moody algebras}
\author[B.~Salisbury]{Ben Salisbury}
\address[B.~Salisbury]{Department of Mathematics, Central Michigan University, Mount Pleasant, MI 48859, USA}
\email{ben.salisbury@cmich.edu}
\urladdr{http://people.cst.cmich.edu/salis1bt/}
\thanks{B.S.\ was partially supported by Simons Foundation grant 429950.}

\author[T.~Scrimshaw]{Travis Scrimshaw}
\address[T.~Scrimshaw]{School of Mathematics and Physics, The University of Queensland, St. Lucia, QLD 4072, Australia}
\email{tcscrims@gmail.com}
\urladdr{https://people.smp.uq.edu.au/TravisScrimshaw/}
\thanks{T.S.\ was partially supported by Australian Research Council DP170102648.}

\begin{abstract}
We construct a uniform model for highest weight crystals and $B(\infty)$ for generalized Kac--Moody algebras using rigged configurations.
We also show an explicit description of the $\ast$-involution on rigged configurations for $B(\infty)$: that the $\ast$-involution interchanges the rigging and the corigging.
We do this by giving a recognition theorem for $B(\infty)$ using the $\ast$-involution.
As a consequence, we also characterize $B(\lambda)$ as a subcrystal of $B(\infty)$ using the $\ast$-involution.
\end{abstract}

\maketitle


\section{Introduction}


Generalized Kac--Moody algebras, also known as Borcherds algebras, are infinite-dimensional Lie algebras introduced by Borcherds~\cite{B88,B92} as a result of his study of the ``Monstrous Moonshine'' conjectures of Conway and Norton~\cite{CN79}.  For more information, see, for example,~\cite{Jurisich96}.

With respect to a symmetrizable Kac--Moody algebra $\g$, crystal bases are combinatorial analogues of representations of the quantized universal enveloping algebra of $\g$.  Defined simultaneously by Kashiwara~\cite{K90,K91} and Lusztig~\cite{Lusztig90} in the early 1990s, crystals have become an integral part of combinatorial representation theory and have seen application to algebraic combinatorics, mathematical physics, the theory of automorphic forms, and more.  In~\cite{JKK05}, Kashiwara's construction of the crystal basis was extended to the symmetrizable generalized Kac--Moody algebra setting.  In particular, the crystal basis for the negative half of the quantized universal enveloping algebra $U_q(\g)$ was introduced, denoted $B(\infty)$, and the crystal basis for the irreducible highest weight module $V(\lambda)$ was also introduced, denoted $B(\lambda)$.  The general combinatorial properties of these crystals were then abstracted in~\cite{JKKS07}, much in the same way that Kashiwara had done in~\cite{K93} for the classical case.  There, theorems characterizing the crystals $B(\infty)$ and $B(\lambda)$ were also proved.  More recently, other combinatorial models for crystals of generalized Kac--Moody algebras are known: Nakajima monomials~\cite{JKKS08}, Littelmann's path model~\cite{JoLa09}, the polyhedral model~\cite{Shin08,Shin08II}, and irreducible components of quiver varieties~\cite{KKS09,KKS12}. Furthermore, there is an extension of Khovanov--Lauda--Rouquier (KLR) algebras for generalized Kac--Moody algebras~\cite{KOP12}.

This paper aims to achieve analogous results to~\cite{SalS15,SalS17,SalS16II} for the case in which $\g$ is a generalized Kac--Moody algebra; that is, to develop a rigged configuration model for the infinity crystal $B(\infty)$, including the $*$-crystal operators, and the irreducible highest weight crystals $B(\lambda)$ when the underlying algebra is a generalized Kac--Moody algebra.  In order to do this, a new recognition theorem (see Theorem~\ref{thm:Binf_recog}) for $B(\infty)$, mimicking the recognition theorem in the classical Kac--Moody cases by Tingley--Webster~\cite[Prop.~1.4]{TW16} (which is a reformulation of~\cite[Prop.~3.2.3]{KS97}), is presented.  The major difference in this new recognition theorem is the existence of imaginary simple roots; the crystal operators associated with imaginary simple roots behave inherently different than that of the case of only real simple roots.  Once the new recognition theorem is established, we state new crystal operators (see Definition~\ref{def:RC_crystal_ops}) and the $*$-crystal operators (see Definition~\ref{def:RC_star_crystal_ops}) on rigged configurations. We then appeal to the fact that $B(\lambda)$ naturally injects into $B(\infty)$ by~\cite[Thm.~5.2]{JKKS07}. We also give a characterization of $B(\lambda)$ inside of $B(\infty)$ using the $\ast$-involution analogous to~\cite[Prop.~8.2]{K95} (see Corollary~\ref{cor:classification_hw_star}).

We note that our results give the first model for crystals of generalized Kac--Moody algebras that has a direct combinatorial description of the $*$-involution on $B(\infty)$; \textit{i.e.}, by not recursively using the crystal and $*$-crystal operators. Moreover, the rigged configuration model for $B(\lambda)$ does not require knowledge other than the combinatorial description of the element, in contrast to the Littelmann path or Nakajima monomial models.

This paper is organized as follows.
In Section~\ref{sec:background}, we give the necessary background on generalized Kac--Moody algebras and their crystals.
In Section~\ref{sec:recognition_thm}, we present the recognition theorem for $B(\infty)$ using the $\ast$-involution.
In Section~\ref{sec:RC}, we construct the rigged configuration model for $B(\infty)$ and the $\ast$-involution.
In Section~\ref{sec:purely_imaginary}, a characterization of rigged configurations belonging to $B(\infty)$ in the purely imaginary case is given.
In Section~\ref{sec:highest_weight}, we show how the rigged configuration model yields highest weight crystals.

\subsection*{Acknowledgements}

TS would like to thank Central Michigan University for its hospitality during his visit in November, 2018, where part of this work was done. TS also would like to thank the Center for Applied Mathematics at Tianjin University for the great working environment during his visit in December, 2018.

\section{Quantum generalized Kac--Moody algebras and crystals}
\label{sec:background}

Let $I$ be a countable set.  A \defn{Borcherds--Cartan matrix} $A = (A_{ab})_{a,b\in I}$ is a real matrix such that
\begin{enumerate}
\item $A_{aa} = 2$ or $A_{aa} \le 0$ for all $a\in I$,
\item $A_{ab} \le 0$ if $i\neq j$, 
\item $A_{ab} \in \ZZ$ if $A_{aa} = 2$, and
\item $A_{ab} = 0$ if and only if $A_{ba} =0$.
\end{enumerate}
An index $a\in I$ is called \defn{real} if $A_{aa} = 2$ and is called \defn{imaginary} if $A_{aa} \le 0$.  The subset of $I$ of all real (resp.\ imaginary) indices is denoted $I^\re$ (resp.\ $I^\im$).  We will always assume that $A_{ab} \in \ZZ$, $A_{aa} \in 2\ZZ_{\le 1}$, and that $A$ is symmetrizable.  Additionally, if $I = I^{\im}$, then the corresponding Borcherds--Cartan matrix will be called \defn{purely imaginary}.


\begin{ex}
Let $I = \{ (i,t) : i \in \ZZ_{\ge -1}, \ 1 \le t \le c(i)\}$, where $c(i)$ is the $i$-th coefficient of the elliptic modular function
\[
j(q) - 744 = q^{-1} + 196884q + 21493760q^2 + \cdots = \sum_{i \ge -1} c(i) q^i.
\]
Define $A = (A_{(i,t),(j,s)})$, where each entry is defined by $A_{(i,t),(j,s)} = -(i+j)$.  This is a Borcherds--Cartan matrix, and it is associated to the Monster Lie algebra used by Borcherds in \cite{B92}.  This matrix is not purely imaginary because $I^\re = \{(-1,1)\}$.
\end{ex}

A \defn{Borcherds--Cartan datum} is a tuple $(A,P^\vee,P,\Pi^\vee,\Pi)$ where
\begin{enumerate}
\item $A$ is a Borcherds--Cartan matrix,
\item $P^\vee = (\bigoplus_{a\in I} \ZZ h_a) \oplus (\bigoplus_{a\in I} \ZZ d_a)$, called the \defn{dual weight lattice},
\item $P = \{ \lambda \in \mathfrak{h}^* : \lambda(P^\vee)\subset \ZZ \}$, where $\mathfrak{h}^* = \QQ\otimes_\ZZ P^\vee$, called the \defn{weight lattice},
\item $\Pi^\vee = \{ h_a : a\in I \}$, called the set of \defn{simple coroots}, and
\item $\Pi = \{ \alpha_a : a\in I \}$, called the set of \defn{simple roots}.
\end{enumerate}
Define the canonical pairing $\langle\ ,\ \rangle\colon P^\vee\times P \longrightarrow \ZZ$ by $\langle h_a, \alpha_b \rangle = A_{ab}$ for all $a,b\in I$.

The set of \defn{dominant integral weights} is $P^+ = \{ \lambda\in P : \lambda(h_a) \ge 0 \text{ for all } a\in I\}$.  The \defn{fundamental weights}, denoted $\Lambda_a \in P^+$ for $a\in I$, are defined by $\langle h_b,\Lambda_a \rangle = \delta_{ab}$ and $\langle d_b, \Lambda_a \rangle = 0$ for all $a,b\in I$.  Finally, set $Q = \bigoplus_{a\in I} \ZZ \alpha_a$ and $Q^- = \sum_{a\in I} \ZZ_{\le0} \alpha_a$. 

Let $U_q(\g)$ be the \defn{quantum generalized Kac--Moody algebra} associated with the Borcherds--Cartan datum $(A,P^\vee,P,\Pi^\vee,\Pi)$.  (For more detailed information on $U_q(\g)$, see, for example, \cite{JKK05}.)

\begin{dfn}[See~\cite{JKKS07}]
An \defn{abstract $U_q(\g)$-crystal} is a set $B$ together with maps
\[
e_a,f_a\colon B \longrightarrow B\sqcup\{\zero\}, \ \ \ \
\varepsilon_a,\varphi_a\colon B \longrightarrow \ZZ\sqcup\{-\infty\}, \ \ \ \ 
\wt\colon B \longrightarrow P,
\]
subject to the following conditions:
\begin{enumerate}
\item $\wt(e_av) = \wt(v) + \alpha_a$ if $e_av \neq \zero$,
\item $\wt(f_av) = \wt(v) - \alpha_a$ if $f_av \neq \zero$,
\item for any $a\in I$ and $v\in B$, $\varphi_a(v) = \varepsilon_a(v) + \langle h_a , \wt(v) \rangle$,
\item for any $a\in I$ and $v,v' \in B$, $f_av = v'$ if and only if $v = e_av'$,
\item for any $a \in I$ and $v\in B$ such that $e_av \neq \zero$, we have
\begin{enumerate}
\item $\varepsilon_a(e_av) = \varepsilon_a(v)-1$ and $\varphi_a(e_av) = \varphi_a(v) + 1$ if $a \in I^\re$,
\item $\varepsilon_a(e_av) = \varepsilon_a(v)$ and $\varphi_a(e_av) = \varphi_a(v) + A_{aa}$ if $a \in I^\im$,
\end{enumerate}
\item for any $a \in I$ and $v\in B$ such that $f_av \neq \zero$, we have
\begin{enumerate}
\item $\varepsilon_a(f_av) = \varepsilon_a(v)+1$ and $\varphi_a(f_av) = \varphi_a(v) - 1$ if $a \in I^\re$,
\item $\varepsilon_a(f_av) = \varepsilon_a(v)$ and $\varphi_a(f_av) = \varphi_a(v) - A_{aa}$ if $a \in I^\im$,
\end{enumerate}
\item for any $a\in I$ and $v\in B$ such that $\varphi_a(v) = -\infty$, we have $e_av = f_av = \zero$.
\end{enumerate}
Here, $\zero$ is considered to be a formal object; i.e., it is not an element of a crystal.
\end{dfn}

\begin{ex}
For each $\lambda \in P^+$, by \cite[\S3]{JKK05}, there exists a unique irreducible highest weight $U_q(\g)$-module $V(\lambda)$ in the category $\Oint$.  (See \cite{JKK05} for the details and explanation of the notation.)  
Associated to each $V(\lambda)$ is a crystal basis $\bigl(L(\lambda),B(\lambda)\bigr)$, in the sense of \cite{JKK05}.  Then $B(\lambda)$ is an abstract $U_q(\g)$-crystal.  In this case, for all $a\in I$ and $v\in B(\lambda)$, we have
\begin{align*}
\varepsilon_a(v) &=
\begin{cases}
 \max\{ k \ge0 : e_a^k v \neq \zero \} & \text{if } a\in I^\re, \\
 0 & \text{if } a \in I^\im,
\end{cases} \\
\varphi_a(v) &=
\begin{cases}
 \max\{ k \ge 0 : f_a^k v \neq \zero \} & \text{if } a \in I^\re, \\
 \langle h_a , \wt(v) \rangle & \text{if } a \in I^\im.
\end{cases}
\end{align*}
Moreover, there exists a unique $u_\lambda \in B(\lambda)$ such that $\wt(u_\lambda) = \lambda$ and 
\[
B(\lambda) = \{ f_{a_1} \cdots f_{a_r} u_\lambda : r \ge 0,\ a_1,\dots,a_r \in I \} \setminus \{\zero\}.
\]
\end{ex}

\begin{ex}
The negative half of the generalized quantum algebra $U_q^-(\g)$ has a crystal basis $\bigl( L(\infty),B(\infty)\bigr)$ in the sense of \cite{JKK05}.   Then $B(\infty)$ is an abstract $U_q(\g)$-crystal.  In this case, there exists a unique element $\one \in B(\infty)$ such that $\wt(\one) = 0$ and 
\[
B(\infty) = \{ f_{a_1} \cdots f_{a_r} \one : r \ge 0,\ a_1,\dots,a_r \in I \}.
\]
Moreover, for all $v\in B(\infty)$ and $a,a_1,\dots,a_r \in I$, we have
\begin{subequations}
\label{eq:B_inf_structure}
\begin{align}
\label{eq:B_inf_ep} \varepsilon_a(v) &= 
\begin{cases}
 \max\{ k\ge 0 : e_a^kv \neq \zero \} & \text{if } a\in I^\re,\\
 0 & \text{if } a \in I^\im,
\end{cases} \\
\label{eq:B_inf_phi} \varphi_a(v) &= \varepsilon_a(v) + \langle h_a, \wt(v)\rangle , \\
\wt(v) &= -\alpha_{a_1} - \cdots - \alpha_{a_r}\ \ \text{ if } v = f_{a_1} \cdots f_{a_r}\one.
\end{align}
\end{subequations}
\end{ex}

\begin{dfn}[See \cite{JKKS07}]
Let $B_1$ and $B_2$ be abstract $U_q(\g)$-crystals.  A \defn{crystal morphism} $\psi\colon B_1 \longrightarrow B_2$ is a map $B_1\sqcup\{\zero\} \longrightarrow B_2\sqcup\{\zero\}$ such that
\begin{enumerate}
\item for $v \in B_1$ and all $a\in I$, we have
\[
\varepsilon_a\bigl( \psi(v) \bigr) = \varepsilon_a(v), \ \ \ \
\varphi_a\bigl( \psi(v) \bigr) = \varphi_a(v),\ \ \ \ 
\wt\bigl( \psi(v) \bigr) = \wt(v) , 
\]
\item if $v \in B_1$ and $f_av \in B_1$, then $\psi(f_av) = f_a\psi(v)$.
\end{enumerate}
\end{dfn}

Let $\psi\colon B_1 \longrightarrow B_2$ be a crystal morphism.  Then $\psi$ is called \defn{strict} if $\psi(e_av) = e_a\psi(v)$ and $\psi(f_av) = f_a\psi(v)$ for all $a \in I$.  The morphism $\psi$ is an \defn{embedding} if the underlying map is injective.  An \defn{isomorphism} of crystals is a bijective, strict crystal morphism.

\begin{dfn}[See \cite{JKKS07}]
Let $B_1$ and $B_2$ be abstract $U_q(\g)$-crystals.  The \defn{tensor product} $B_1\otimes B_2$ is a crystal with operations defined, for $a\in I$, by
\begin{align*}
 e_a(v_1\otimes v_2) &= 
 \begin{cases} 
  e_av_1 \otimes v_2 & \text{if } a\in I^\re \text{ and }\varphi_a(v_1)\ge\varepsilon_a(v_2), \\
  e_av_1 \otimes v_2 & \text{if } a\in I^\im \text{ and } \varphi_a(v_1) > \varepsilon_a(v_2) - A_{aa}, \\
  \zero & \text{if } a\in I^\im \text{ and } \varepsilon_a(v_2) < \varphi_a(v_1) \le \varepsilon_a(v_2)-A_{aa}, \\
  v_1 \otimes e_av_2 & \text{if } a\in I^\re \text{ and }\varphi_a(v_1) < \varepsilon_a(v_2), \\
  v_1 \otimes e_av_2 & \text{if } a\in I^\im \text{ and } \varphi_a(v_1) \le \varepsilon_a(v_2),
 \end{cases} \\
 f_a(v_1\otimes v_2) &= 
 \begin{cases}
  f_av_1 \otimes v_2 & \text{if } \varphi_a(v_1) > \varepsilon_a(v_2) , \\
  v_1 \otimes f_av_2 & \text{if } \varphi_a(v_1) \le \varepsilon_a(v_2), 
 \end{cases}\\
 \varepsilon_a(v_1\otimes v_2) &= \max\bigl\{ \varepsilon_a(v_1), \varepsilon_a(v_2) - \langle h_a, \wt(v_1) \rangle \bigr\} , \\
 \varphi_a(v_1\otimes v_2) &= \max\bigl\{ \varphi_a(v_1) + \langle h_a, \wt(v_2) \rangle, \varphi_a(v_2) \bigr\}, \\
 \wt(v_1\otimes v_2) &= \wt(v_1) + \wt(v_2).
\end{align*}
\end{dfn}

\begin{ex}
\label{ex:T_crystal}
Let $\lambda \in P$ and set $T_\lambda = \{ t_\lambda\}$.  For all $a\in I$, define crystal operations
\[
e_at_\lambda = f_at_\lambda = \zero , \ \ \ \ 
\varepsilon_a(t_\lambda) = \varphi_a(t_\lambda) = -\infty, \ \ \ \ 
\wt(t_\lambda) = \lambda.
\]
Note that $T_\lambda \otimes T_\mu \cong T_{\lambda+\mu}$, for $\lambda,\mu\in P$.  Moreover, by \cite[Prop.~3.9]{JKKS07}, for every $\lambda \in P^+$, there exists a crystal embedding $\iota_\lambda \colon B(\lambda) \lhook\joinrel\longrightarrow B(\infty)\otimes T_\lambda$.
\end{ex}

\begin{ex}
Let $C = \{c\}$.  Then $C$ is a crystal with operations defined, for $a\in I$, by
\[
e_ac = f_ac = \zero, \ \ \ \ 
\varepsilon_a(c) = \varphi_a(c) = 0, \ \ \ \
\wt(c) = 0.
\]
\end{ex}

\begin{thm}[See {\cite[Thm.~5.2]{JKKS07}}]
\label{thm:cutout}
Let $\lambda \in P^+$.  Then $B(\lambda)$ is isomorphic to the connected component of $B(\infty)\otimes T_\lambda \otimes C$ containing $\one \otimes t_\lambda \otimes c$.  
\end{thm}

\begin{ex}
For each $a\in I$, set $\NN_{(a)} = \{ z_a(-n) : n\ge 0 \}$.  Then $\NN_{(a)}$ is a crystal with maps defined, for $b\in I$, by
\begin{gather*}
e_bz_a(-n) = 
\begin{cases}
 z_a(-n+1) & \text{if } b=a,\\
 \zero & \text{otherwise},
\end{cases}
\ \ \ \ \ 
f_bz_a(-n) = 
\begin{cases}
 z_a(-n-1) & \text{if } b=a,\\
 \zero & \text{otherwise},
\end{cases} 
\\ 
\varepsilon_b \bigl( z_a(-n) \bigr) =
\begin{cases}
 n & \text{if } b=a \in I^\re,\\
 0 & \text{if } b=a \in I^\im,\\
 -\infty & \text{otherwise},
\end{cases}
\ \ \ \ \ 
\varphi_b \bigl( z_a(-n) \bigr) =
\begin{cases}
 -n & \text{if } b=a \in I^\re,\\
 -nA_{aa} & \text{if } b=a \in I^\im,\\
 -\infty & \text{otherwise},
\end{cases}\\
\wt\bigl( z_a(-n) \bigr) = -n\alpha_a.
\end{gather*}
By convention, $z_a(-n) = \zero$ for $n < 0$.
\end{ex}

\begin{thm}[See {\cite[Thm.~4.1]{JKKS07}}]
For any $a\in I$, there exists a unique strict crystal embedding $B(\infty) \lhook\joinrel\longrightarrow B(\infty) \otimes \NN_{(a)}$.  
\end{thm}

\section{Recognition theorem for $B(\infty)$}
\label{sec:recognition_thm}

\begin{thm}[See {\cite[Thm.~5.1]{JKKS07}}]
\label{thm:JKKS07infrec}
Let $B$ be an abstract $U_q(\g)$-crystal such that
\begin{enumerate}
\item\label{item:rec_1} $\wt(B) \subseteq Q^-$,
\item\label{item:rec_2} there exists an element $v_0 \in B$ such that $\wt(v_0) = 0$,
\item\label{item:rec_3} for any $v \in B$ such that $v \neq v_0$, there exists some $a\in I$ such that $e_av \neq \zero$, and
\item\label{item:rec_4} for all $a\in I$, there exists a strict embedding $\Psi_a\colon B \lhook\joinrel\longrightarrow B\otimes \NN_{(a)}$.
\end{enumerate}
Then there exists a crystal isomorphism $B \cong B(\infty)$ such that $v_0 \mapsto \one$.
\end{thm}

There is a $\QQ(q)$-antiautomorphism $*\colon U_q(\g) \longrightarrow U_q(\g)$ defined by
\[
E_a \mapsto E_a, \qquad
F_a \mapsto F_a, \qquad
q \mapsto q, \qquad
q^h \mapsto q^{-h},
\]
where $E_a$, $F_a$, and $q^h$ ($a\in I$, $h \in P^\vee$) are the generators for $U_q(\g)$ (see \cite[\S6]{JKK05}).  This is an involution which leaves $U_q^-(\g)$ stable.  Thus, the map $\ast$ induces a map on $B(\infty)$, which we also denote by $*$, and is called the \defn{$\ast$-involution} or \defn{star involution} (and is sometimes known as Kashiwara's involution \cite{CT15,JL09,Kamnitzer07,L12,Sav09,TW16}).  Denote by $B(\infty)^*$ the image of $B(\infty)$ under $*$.

\begin{thm}[See {\cite[Thm.~4.7]{L12}}]
\label{thm:reg_iso_star}
We have
$
B(\infty)^* = B(\infty).
$
\end{thm}

This induces a new crystal structure on $B(\infty)$ with Kashiwara operators
\[
 e_a^* = * \circ  e_a \circ *,
 \qquad\qquad
 f_a^* = * \circ  f_a \circ *,
\]
and the remaining crystal structure is given by
\[
\varepsilon_a^* = \varepsilon_a \circ * , \qquad \qquad
\varphi_a^* = \varphi_a \circ *,
\]
and weight function $\wt$, the usual weight function on $B(\infty)$. From~\cite{L12}, we can combinatorially define $e_a^*$ and $f_a^*$ by
\[
e_a^*v = \Psi_a^{-1}\bigl( v' \otimes z_a(-k+1) \bigr),
\qquad\qquad
f_a^*v = \Psi_a^{-1}\bigl( v' \otimes z_a(-k-1) \bigr),
\]
where $\Psi_a(v) = v' \otimes z_a(-k)$.

We will also need the modified statistics:
\begin{align*}
\tep_a(v) & := \max\{ k' \ge 0 \mid e_a^{k'} v \neq \zero \},
\\ \widetilde{\varphi}_a(v) & := \max\{ k' \ge 0 \mid f_a^{k'} v \neq \zero \},
\end{align*}
and similarly for $\tep^{*}_a$ and $\widetilde{\varphi}^{*}_a$ using $e_a^{*}$ and $f_a^{*}$ respectively.
Note that $\tep_a(v) = \varepsilon_a(v)$ and $\widetilde{\varphi}_a(v) = \varphi_a(v)$, as well as for the $*$-versions, when $a \in I^{\re}$.
Additionally, for $v\in B(\infty)$ and $a\in I$, define
\begin{equation}
\label{eq:jump}
\kappa_a(v) := \begin{cases}
\varepsilon_a(v) + \varepsilon_a^*(v) + \bigl\langle h_a, \wt(v) \bigr\rangle & \text{if } a \in I^{\re}, \\
\varepsilon_a(v) + \tep^{*}_a(v) A_{aa} + \bigl\langle h_a, \wt(v) \bigr\rangle & \text{if } a \in I^{\im}.
\end{cases}
\end{equation}

We will appeal to the following statement, which is a generalized Kac--Moody analogue of the result used in~\cite{SalS16II} coming from~\cite{TW16} (but based on Kashiwara and Saito's classification theorem for $B(\infty)$ in the Kac--Moody setting from~\cite{KS97}).  First, a \defn{bicrystal} is a set $B$ with two abstract $U_q(\g)$-crystal structures $(B,e_a,f_a,\varepsilon_a,\varphi_a,\wt)$ and $(B,e_a^\star,f_a^\star,\varepsilon_a^\star,\varphi_a^\star,\wt)$ with the same weight function.  In such a bicrystal $B$, we say $v\in B$ is a \defn{highest weight element} if $e_av = e_a^\star v = \zero$ for all $a \in I$.

\begin{thm}
\label{thm:Binf_recog}
Let $(B,e_a,f_a,\varepsilon_a,\varphi_a,\wt)$ and $(B^\star,e_a^\star,f_a^\star,\varepsilon_a^\star,\varphi_a^\star,\wt)$ be connected abstract $U_q(\g)$-crystals with the same highest weight element $v_0 \in B \cap B^\star$ that is the unique element of weight $0$, where the remaining crystal data is determined by setting $\wt(v_0) = 0$ and $\varepsilon_a(v)$ by Equation~\eqref{eq:B_inf_ep}. Assume further that, for all $a \neq b$ in $I$ and all $v\in B$, 
\begin{enumerate}
\item\label{item:star1} $f_av$, $f_a^\star v \neq \zero$;
\item\label{item:star2} $f_a^\star f_bv = f_bf_a^\star v$ and $\tep_a^{\star}(f_b v) = \tep_a^{\star}(v)$ and $\tep_b(f_a^{\star} v) = \tep_b(v)$; 
\item\label{item:star3} $\kappa_a(v) = 0$ implies $f_av = f_a^\star v$;
\item\label{item:star4} for $a \in I^{\re}$:
  \begin{enumerate}
  \item\label{item:star4a} $\kappa_a(v) \ge 0$;
  \item\label{item:star4b} $\kappa_a(v) \ge 1$ implies $\varepsilon_a^\star(f_av) = \varepsilon_a^\star(v)$ and $\varepsilon_a(f_a^\star v) = \varepsilon_a(v)$;
  \item\label{item:star4c} $\kappa_a(v) \ge 2$ implies $f_af_a^\star v = f_a^\star f_av$;
  \end{enumerate}
\item\label{item:star5} for $a \in I^{\im}$: $\kappa_a(v) > 0$ implies $\tep_a^{\star}(f_a v) = \tep_a^{\star}(v)$ and $f_a f_a^{\star} v = f_a^{\star} f_a v$.
\end{enumerate}
Then
$
(B,e_a,f_a,\varepsilon_a,\varphi_a,\wt) \iso B(\infty).
$
Moreover, suppose $\kappa_a(v) = 0$ if and only if
\[
\kappa_a^{\star}(v) := \varepsilon_a^*(v) + \tep_a(v) A_{aa} + \bigl\langle h_a, \wt(v) \bigr\rangle = 0
\]
for all $a \in I^{\im}$ and $v \in B$.
Then
\[
(B^\star,e_a^\star,f_a^\star,\varepsilon_a^\star,\varphi_a^\star,\wt) \iso B(\infty)
\]
with $e_a^\star = e_a^*$ and $f_a^\star = f_a^*$.
\end{thm}

\begin{proof}
We will show our conditions are equivalent for $(B,e_a,f_a,\varepsilon_a,\varphi_a,\wt)$ to those of Theorem~\ref{thm:JKKS07infrec}, and the claim $(B,e_a,f_a,\varepsilon_a,\varphi_a,\wt) \iso B(\infty)$ follows by a similar proof to~\cite[Prop.~2.3]{SalS16II}.

We first assume the conditions of Theorem~\ref{thm:JKKS07infrec} hold for $B$. It is straightforward to see $v_0$ exists. The map $\Psi_a \colon B \longrightarrow B \otimes \NN_{(a)}$ defined by
\begin{equation}
\label{eq:defining_star}
\Psi_a(v) = (e_a^{\star})^k v \otimes f_a^k z_a(0) = v' \otimes z_a(-k),
\end{equation}
where $0 \leq k := \tep_a(v)$, is a strict crystal embedding by our assumptions.
Conditions~(\ref{item:star1}) and~(\ref{item:star2}) follow from the tensor product rule and the definition of $f_a^{\star}$.
The remaining conditions were shown in~\cite[Prop.~1.4]{TW16}\footnote{We have to take the dual crystal and corresponding dual properties, see~\cite[Prop.~2.2]{SalS16II}.} and~\cite[Lemma~4.2]{L12}.


Next, we assume Conditions~(\ref{item:star1}--\ref{item:star5}) hold. We have $B = B^{\star}$ by a similar argument to~\cite[Prop.~2.3]{SalS16II}. Next, we construct a strict crystal embedding $\Psi_a \colon B \lhook\joinrel\longrightarrow B \otimes \NN_{(a)}$. We begin by defining a map $\Psi_a$ by Equation~\eqref{eq:defining_star}.
If $\Psi_a$ is a strict crystal morphism, then we have $\Psi_a$ is an embedding by induction on depth using that $B$ is generated from $v_0$, that $\Psi_a(v_0) = v_0 \otimes z_a(0)$, and that $v_0 \otimes z_a(0)$ is the unique element of weight $0$ in $B \otimes \NN_{(a)}$. Thus, it is sufficient to show that $\Psi_a$ is a strict crystal morphism.

Assume $a \neq b$. Since $\tep_a^{\star}(f_b v) = \tep_a^{\star}(v)$ by Condition~(\ref{item:star2}), then we have $e_a^{\star} v \neq \zero$ if and only if $e_a^{\star} f_b v \neq \zero$. Thus, if $e_a^{\star} v \neq \zero$, we have
\begin{equation}
\label{eq:fe*_commute}
f_b e_a^{\star} v = e_a^{\star} f_a^{\star} f_b e_a^{\star} v = e_a^{\star} f_b f_a^{\star} e_a^{\star} v = e_a^{\star} f_b v
\end{equation}
since $e_a^{\star} f_a^{\star} w = w$ for all $w \in B$ by Condition~(\ref{item:star1}) and the crystal axioms.
Similarly, if $e_a^{\star} e_b v \neq \zero$ (or $e_b e_a^{\star} v \neq \zero$), then we have
\begin{equation}
\label{eq:ee*_commute}
e_a^{\star} e_b v = e_a^{\star} e_b f_a^{\star} e_a^{\star} v = e_a^{\star} e_b f_a^{\star} f_b e_b e_a^{\star} v = e_a^{\star} e_b f_b f_a^{\star} e_b e_a^{\star} v = e_b e_a^{\star} v.
\end{equation}
Note that $\tep_a^{\star}(f_b v) = \tep_a^{\star}(v)$ implies $\tep_a^{\star}(e_b v) = \tep_a^{\star}(v)$ by the crystal axioms, and so we cannot have $e_a^{\star} v = \zero$ and $e_a^{\star} e_b v \neq \zero$.
Therefore, by the tensor product rule, we have
\begin{align*}
\Psi_a(f_b v) & = (e_a^{\star})^k f_b v \otimes z_a(-k)
\\ &= (e_a^{\star})^k f_b (f_a^{\star})^k (e_a^{\star})^k v \otimes z_a(-k)
\\ &= (e_a^{\star})^k (f_a^{\star})^k f_b (e_a^{\star})^k v \otimes z_a(-k)
\\ &= f_b (e_a^{\star})^k v \otimes z_a(-k) 
\\ &= f_b \Psi_a(v)
\end{align*}
and
\[
\Psi_a(e_b v) = (e_a^{\star})^k e_b v \otimes z_a(-k) = e_b (e_a^{\star})^k v \otimes z_a(-k) = e_b \Psi_a(v).
\]
For $a \in I^{\re}$, we have $f_a \Psi_a(v) = \Psi_a(f_a v)$ and $e_a \Psi_a(v) = \Psi_a(e_a v)$ by~\cite[Prop.~1.4]{TW16}.

Hence, we assume $a \in I^{\im}$. We note that
\begin{align*}
\kappa_a(v) &= 0 + k A_{aa} + \langle h_a, \wt(v) \rangle 
\\ &= \bigl\langle h_a, \wt(v) + k\alpha_a \bigr\rangle 
\\ &= \bigl\langle h_a, \wt\bigl( (e_a^{\star})^k v \bigr) \bigr\rangle
\\ &= \varphi_a(v') 
\\ &= \bigl\langle h_a, \wt(v') \bigr\rangle 
\\ &\geq 0.
\end{align*}
By the tensor product rule, we have
\[
f_a \Psi_a(v) = f_a \bigl( v' \otimes z_a(-k) \bigr) =
\begin{cases}
v' \otimes f_a z_a(-k)  & \text{if } \varphi_a(v') = 0, \\
f_a(v') \otimes z_a(-k) & \text{if } \varphi_a(v') > 0.
\end{cases}
\]
We first consider $\kappa_a(v) = 0 = \varphi_a(v')$. Note that $f_a = f_a^{\star}$ implies $\tep_a^{\star} \bigl( f_a v \bigr) = k + 1$ and $(e_a^{\star})^{k+1}(f_a v) = v'$. Therefore, we have $f_a \Psi_a(v) = \Psi_a(f_a v)$ by the definition of $\Psi_a$.
Next, assume $\kappa_a(v) = \varphi_a(v') > 0$, and we note that
\begin{align*}
\kappa_a(e_a^{\star} v) &= A_{aa} \tep_a^{\star}(e_a^{\star} v) + \bigl\langle h_a, \wt(e_a^{\star} v) \bigr\rangle
\\ & = A_{aa} \bigl(\tep_a^{\star}(v) - 1 \bigr) + \bigl\langle h_a, \wt(v) \bigr\rangle + A_{aa} 
\\ & = \kappa_a(v).
\end{align*}
Thus, we have
\[
\Psi_a(f_a v) = (e_a^{\star})^k f_a v \otimes z_a(-k) = f_a (e_a^{\star})^k v \otimes z_a(-k) = f_a \Psi_a(v)
\]
by $\tep_a^{\star}(f_a v) = \tep_a^{\star}(v)$ and Equation~\eqref{eq:fe*_commute} with $b = a$.

Again, by the tensor product rule, we have
\[
e_a\Psi_a(v) = e_a\bigl( v' \otimes z_a(-k) \bigr) = 
\begin{cases}
e_av' \otimes z_a(-k) & \text{if } \varphi_a(v')  > -A_{aa} , \\
\zero & \text{if } 0 < \varphi_a(v') \le -A_{aa}, \\
v' \otimes z_a(-k+1) & \text{if } \varphi_a(v') \le 0.
\end{cases}
\]
If $\kappa_a(v) = \varphi_a(v') = 0$, then $e_a = e_a^{\star}$ and $e_a \Psi_a(v) = \Psi_a(e_a v)$ by the construction of $\Psi_a$ and noting in this case $e_a v = 0$ if and only if $k = 0$. Next, suppose $\kappa_a(v) = \varphi_a(v') > -A_{aa}$, and so we have
\begin{align*}
\Psi_a(e_a v) & = (e_a^{\star})^k e_a v \otimes z_a(-k)
= e_a (e_a^{\star})^k v \otimes z_a(-k) = e_a \Psi_a(v)
\end{align*}
by $\tep_a^{\star}(e_a v) = \tep_a^{\star}(v)$, which follows from Condition~(\ref{item:star5}) and the crystal axioms, and Equation~\eqref{eq:ee*_commute} with $b = a$. Finally, consider the case $0 < \kappa_a(v) = \varphi_a(v') \leq -A_{aa}$. If $e_a v \neq \zero$, then we have
\begin{align*}
\kappa_a(e_a v) & = \tep^{\star}_a(e_a v) A_{aa} + \bigl\langle h_a, \wt(e_a v) \bigr\rangle
\\ &= k A_{aa} + \bigl\langle h_a, \wt(v) \bigr\rangle + A_{aa}
\\ & = \kappa_a(v) + A_{aa} 
\\ &\leq -A_{aa} + A_{aa} 
\\ & = 0,
\end{align*}
where $\tep_a^{\star}(e_a v) = k$ by Condition~(\ref{item:star5}). Since we must have $\kappa_a(w) \geq 0$ for all $w \in B$, we must have $\kappa_a(e_a v) = 0$. Hence, by Condition~(\ref{item:star3}), we have $f_a^{\star} e_a v = f_a e_a v = v$, which implies that $e_a = e_a^{\star}$ and $(e_a^{\star})^{k+1} \neq \zero$. However, this is a contradiction since $(e_a^{\star})^{k+1} v = \zero$ by the definition of $\tep_a^{\star}(v)$. Therefore, we have $e_a \Psi_a(v) = \Psi_a(e_a v)$.

It is straightforward to see that for all $v \in B$, we have
\[
\varepsilon_a\bigl( \Psi_a(v) \bigr) = \varepsilon_a(v),
\qquad
\varphi_a\bigl( \Psi_a(v) \bigr) = \varphi_a(v),
\qquad
\wt\bigl( \Psi_a(v) \bigr) = \wt(v),
\]
from the tensor product rule and the crystal axioms. Thus, $\Psi_a$ is a strict crystal morphism.

Finally, we have that for any $v \in B$, we can write $v = x_{a_1} \cdots x_{a_{\ell}} v_0$, where $a_i \in I$ and $x = e, f$. Since, $\Psi_a$ is a strict crystal morphism, we have
\[
\Psi_{\vec{a}}(v) = \Psi_{\vec{a}}(x_{a_1} \cdots x_{a_{\ell}} v_0) = x_{a_1} \cdots x_{a_{\ell}} \Psi_{\vec{a}}(v_0) \in \{v_0\} \otimes \NN_{(a_1)} \otimes \cdots \otimes \NN_{(a_{\ell})},
\]
where $\Psi_{\vec{a}} = \Psi_{a_1} \circ \cdots \circ \Psi_{a_{\ell}}$. Since $\Psi_a$ is an embedding, we have 
\[
\Psi_{\vec{a}}(v) = v_0 \otimes z_{a_1}(0) \otimes \cdots \otimes z_{a_{\ell}}(0) \text{ if and only if } v = v_0. 
\]
If $v \neq v_0$, then by the tensor product rule, there exists some $b \in I$ such that $\Psi_{\vec{a}}(e_b v) = e_b \Psi_{\vec{a}}(v) \neq \zero$, implying $e_b \neq \zero$. Thus, $v_0$ is the unique highest weight vector of $B$ and $\wt(B) \subseteq Q^-$, and so $(B, e_a, f_a, \varepsilon_a, \varphi_a, \wt) \iso B(\infty)$ follows.

Now additionally suppose $\kappa_a(v) = 0$ if and only if $\kappa_a^{\star}(v) = 0$ for all $a \in I^{\im}$ and $v \in B$. Note that $\kappa_a(f_a^{\star} v) = \kappa_a(v)$, and so $\kappa_a(f_a v) = \kappa_a(v) = 0$ when $\kappa_a(v) = 0$ and
\[
\kappa_a(f_a v) = \kappa_a(v) - A_{aa} \geq \kappa_a(v) > 0
\]
otherwise. Thus, the same conditions of the theorem hold by swapping $e_a$ with $e_a^{\star}$ and $f_a$ with $f_a^{\star}$, and hence
$
(B^\star,e_a^\star,f_a^\star,\varepsilon_a^\star,\varphi_a^\star,\wt) \iso B(\infty).
$
By induction on depth, we have $e_a^* = e_a^{\star}$ and $f_a^* = f_a^{\star}$ by the definition of $e_a^*$ and $f_a^*$.
\end{proof}

\begin{remark}
As the proof of Theorem~\ref{thm:Binf_recog} shows, the conditions given in Theorem~\ref{thm:JKKS07infrec} are actually stronger than needed and can have conditions closer to~\cite[Prop.~3.2.3]{KS97}. Indeed, instead of requiring a unique highest weight element, one can use that there is a unique element of weight $0$ that is \emph{a} highest weight element.
\end{remark}

\begin{remark}
Unlike for $a \in I^{\re}$, the value $\kappa_a$ for $a \in I^{\im}$ does not have the duality under taking the $*$-involution. However, this is expected as the action of $e_a$ and $e_a^*$ are needed to be expressed somewhere in the recognition theorem. In contrast, the action of $e_a$ and $e_a^*$, for $a \in I^{\re}$, was included in the definition of $\varepsilon_a$ and $\varepsilon_a^*$ respectively. Yet, we do obtain the duality by the condition that $\kappa_a(v) = 0$ if and only if $\kappa_a^*(v) = 0$. Additionally, note that $\kappa_a(v) = \kappa_a^*(v)$ for all $a \in I^{\re}$ and $v \in B$.
\end{remark}

\section{Rigged configurations}
\label{sec:RC}

Let $\HH = I \times \ZZ_{>0}$. A rigged configuration is a sequence of partitions $\nu = (\nu^{(a)} \mid a \in I)$ such that each row $\nu_i^{(a)}$ has an integer called a \defn{rigging}, and we let $J = \bigl(J_i^{(a)} \mid (a, i) \in \HH \bigr)$, where $J_i^{(a)}$ is the multiset of riggings of rows of length $i$ in $\nu^{(a)}$. We consider there to be an infinite number of rows of length $0$ with rigging $0$; i.e., $J_0^{(a)} = \{0, 0, \dotsc\}$ for all $a \in I$. The term rigging will be interchanged freely with the term \defn{label}.  We identify two rigged configurations $(\nu, J)$ and $(\widetilde{\nu}, \widetilde{J})$ if 
$\nu = \widetilde{\nu}$ and $J_i^{(a)} = \widetilde{J}_i^{(a)}$
for any fixed $(a, i) \in \HH$. Let $(\nu, J)^{(a)}$ denote the rigged partition $(\nu^{(a)}, J^{(a)})$.

Define the \defn{vacancy numbers} of $\nu$ to be 
\begin{equation}
\label{eq:vacancy}
p_i^{(a)}(\nu) = p_i^{(a)} = - \sum_{(b,j) \in \HH} A_{ab} \min(i, j) m_j^{(b)},
\end{equation}
where $m_i^{(a)}$ is the number of parts of length $i$ in $\nu^{(a)}$ and $(A_{ab})_{a,b\in I}$ is the underlying Borcherds--Cartan matrix.  The \defn{corigging}, or \defn{colabel}, of a row in $(\nu,J)^{(a)}$ with rigging $x$ is $p_i^{(a)} - x$.  In addition, we can extend the vacancy numbers to
\[
p_{\infty}^{(a)} = \lim_{i\to\infty} p_i^{(a)} =  - \sum_{b \in I} A_{ab} \lvert \nu^{(b)} \rvert
\]
since $\sum_{j=1}^{\infty} \min(i,j) m_j^{(b)} = \lvert \nu^{(b)} \rvert$ for $i \gg 1$.  Note this is consistent with letting $i = \infty$ in Equation~\eqref{eq:vacancy}.

Let $\RC(\infty)$ denote the set of rigged configurations generated by $(\nu_{\emptyset}, J_{\emptyset})$, where $\nu_{\emptyset}^{(a)} = 0$ for all $a \in I$, and closed under the operators $e_a$ and $f_a$ $(a\in I)$ defined next.  Recall that, in our convention, $x \le 0$ since there the string $(0,0)$ is in each $(\nu,J)^{(a)}$.

\begin{dfn}
\label{def:RC_crystal_ops}
Fix some $a \in I$. Let $x$ be the smallest rigging in $(\nu, J)^{(a)}$.
\begin{itemize}
\item[\defn{$e_a$}:] We initially split this into two cases:
  \begin{itemize}[leftmargin=.6in]
  \item[$a \in I^{\re}$:]  If $x = 0$, then $e_a(\nu, J) = \zero$. Otherwise, let $r$ be a row in $(\nu, J)^{(a)}$ of minimal length $\ell$ with rigging $x$.
  \item[$a \in I^{\im}$:] If $\nu^{(a)} = \emptyset$ or $x \neq -A_{aa}/2$, then $e_a(\nu, J) = \zero$. Otherwise let $r$ be the row with rigging $-A_{aa}/2$. 
  \end{itemize}
  If $e_a(\nu, J) \neq \zero$, then $e_a(\nu, J)$ is the rigged configuration that removes a box from row $r$, sets the new rigging of $r$ to be $x + A_{aa}/2$, and changes all other riggings such that the coriggings remain fixed.

\item[\defn{$f_a$}:] Let $r$ be a row in $(\nu, J)^{(a)}$ of maximal length $\ell$ with rigging $x$. Then $f_a(\nu, J)$ is the rigged configuration that adds a box to row $r$, sets the new rigging of $r$ to be $x - A_{aa}/2$, and changes all other riggings such that the coriggings remain fixed.
\end{itemize}
\end{dfn}

We note that explicitly, the other riggings $x \in (\nu, J)^{(b)}$ in a row of length $i$ are changed by $f_a$ according to
\[
x' = \begin{cases} x & \text{if } i \leq \ell, \\ x - A_{ab} & \text{if } i > \ell, \end{cases}
\]
and by $e_a$ according to
\[
x' = \begin{cases} x & \text{if } i <\ell, \\ x + A_{ab} & \text{if } i \geq \ell, \end{cases}
\]
where $\ell$ is the length of the row that was changed.

Define the following additional maps on $\RC(\infty)$ by
\begin{align*}
\varepsilon_a(\nu, J) &= 
\begin{cases}
 \max \{ k \in \ZZ \mid e_a^k(\nu, J) \neq 0 \} & \text{if } a \in I^\re \\
 0 & \text{if } a \in I^\im,
\end{cases}\\
\varphi_a(\nu, J) &= \inner{h_a}{\wt(\nu,J)} + \varepsilon_a(\nu, J),\\
\wt(\nu, J) &= -\sum_{a \in I} \lvert \nu^{(a)} \rvert \alpha_a.
\end{align*}
From this structure, we have $p_\infty^{(a)} = \inner{h_a}{\wt(\nu,J)}$ for all $a \in I$.

\begin{lemma}
\label{lemma:imaginary_partition}
Suppose $a \in I^{\im}$ and $(\nu, J) \in \RC(\infty)$. Then $\nu^{(a)} = (1^k)$, for some $k \ge 0$, and $x \geq -A_{aa}/2$ for any string $(i, x)$ such that $i =1$.
\end{lemma}

\begin{proof}
This is a straightforward induction on depth and by the definition of the crystal operators.
\end{proof}

\begin{prop}
\label{prop:RC_is_crystal}
With the operations above, $\RC(\infty)$ is an abstract $U_q(\g)$-crystal.
\end{prop}

\begin{proof}
From~\cite[Lemma~3.3]{SalS15} and the definitions, we only need to show that
\begin{enumerate}
\item\label{item:abstract1} For any $a \in I^{\im}$, we have $e_a f_a(\nu, J) = (\nu, J)$.

\item\label{item:abstract2} If $e_a(\nu, J) \neq \zero$ for some $a \in I^{\im}$, we have $f_a e_a(\nu, J) = (\nu, J)$.
\end{enumerate}

Both of these properties are straightforward from the crystal operators.
\end{proof}

\begin{prop}
\label{prop:ep_phi}
Let $(\nu, J) \in \RC(\infty)$ and fix some $a \in I$. Let $x \le 0$ denote the smallest label in $(\nu,J)^{(a)}$. Then we have
\[
\varepsilon_a(\nu, J) = -x \hspace{40pt} \varphi_a(\nu, J) = p_{\infty}^{(a)} - x.
\]
\end{prop}

\begin{proof}
For $a \in I^{\re}$, this was shown in~\cite{Sakamoto14,SalS15,S06,SchillingS15}. For $a \in I^{\im}$, this follows from Lemma~\ref{lemma:imaginary_partition}.
\end{proof}

%
%
%

\begin{dfn}
\label{def:RC_star_crystal_ops}
Fix some $a \in I$. Let $x$ be the smallest corigging in $(\nu, J)^{(a)}$. 
\begin{itemize}
\item[\defn{$e_a^*$}:] We initially split this into two cases:
  \begin{itemize}[leftmargin=.6in]
  \item[$a \in I^{\re}$:] If $x = 0$, then $e_a^*(\nu, J) = \zero$. Otherwise, let $r$ be a row in $(\nu, J)^{(a)}$ of minimal length $\ell$ with corigging $x$.
  \item[$a \in I^{\im}$:] If $\nu^{(a)} = \emptyset$ or $x\neq -A_{aa}/2$, then $e_a^*(\nu, J) = \zero$. Otherwise let $r$ be the row with corigging $-A_{aa}/2$.
  \end{itemize}
  If $e_a^*(\nu, J) \neq \zero$, then $e_a^*(\nu, J)$ is the rigged configuration that removes a box from row $r$, sets the rigging of $r$ so that the corigging is $x - A_{aa}/2$, and keeps all other riggings fixed.

\item[\defn{$f_a^*$}:] Let $r$ be a row in $(\nu, J)^{(a)}$ of maximal length $\ell$ with corigging $x$. Then $f_a^*(\nu, J)$ is the rigged configuration that adds a box to row $r$, sets the rigging of $r$ so that the corigging is $x - A_{aa}/2$, and keeps all other riggings fixed.
\end{itemize}
\end{dfn}

%

Let $\RC(\infty)^*$ denote the closure of $(\nu_{\emptyset}, J_{\emptyset})$ under $f_a^*$ and $e_a^*$. We define the remaining crystal structure by
\begin{align*}
\varepsilon_a^*(\nu, J) &= 
\begin{cases}
\max \{ k \in \ZZ \mid (e_a^*)^k(\nu, J) \neq 0 \} & \text{if } a \in I^\re, \\
0 & \text{if } a \in I^\im,
\end{cases} \\
\varphi_a^*(\nu, J) &= \inner{h_a}{\wt(\nu,J)} + \varepsilon_a^*(\nu, J),\\
\wt(\nu, J) &= -\sum_{a \in I} \lvert \nu^{(a)} \rvert \alpha_a.
\end{align*}

\begin{remark}
\label{remark:duality}
We will say an argument holds by duality when we can interchange:
\begin{itemize}
\item ``rigging'' and ``corigging'';
\item $e_a$ and $e_a^*$;
\item $f_a$ and $f_a^*$.
\end{itemize}
\end{remark}

The following two statements hold by duality with Proposition~\ref{prop:RC_is_crystal} and Proposition~\ref{prop:ep_phi} respectively.

\begin{prop}
The tuple $(\RC(\infty)^*, e_a^*, f_a^*, \varepsilon_a^*, \varphi_a^*, \wt)$ is an abstract $U_q(\g)$-crystal.
\end{prop}

\begin{prop}
\label{prop:ep_phi_star}
Let $(\nu, J) \in \RC(\infty)$ and fix some $a \in I$. Let $x$ denote the smallest corigging in $(\nu,J)^{(a)}$. Then we have
\[
\varepsilon_a^*(\nu, J) = -\min(0, x), \hspace{40pt} \varphi_a^*(\nu, J) = p_{\infty}^{(a)} - \min(0, x).
\]
\end{prop}

Now we prove our main result.

\begin{thm}
\label{thm:RCisBinf}
As $U_q(\g)$-crystals, $\RC(\infty) \cong B(\infty)$ and
\[
e_a^* = * \circ e_a \circ *,
\qquad\qquad
f_a^* = * \circ f_a \circ *.
\]
\end{thm}

\begin{proof}
We show that the conditions of Theorem~\ref{thm:Binf_recog} hold. By construction, $f_a(\nu, J), f_a^*(\nu, J) \neq \zero$. The proof that $f_a^* f_b (\nu, J) = f_b f_a^*(\nu, J)$ for all $a \neq b$ follows from the fact that $f_a^*$ (resp.\ $f_b$) preserves riggings (resp.\ coriggings).\footnote{This proof is analogous to the proof for when $a,b \in I^{\re}$ given in~\cite[Thm.~4.13]{SalS16II}.} By~\cite[Thm.~4.13]{SalS16II}, it is sufficient to prove the remaining conditions hold for $a \in I^{\im}$.

Fix some $a \in I^{\im}$. Let $(\nu, J) \in \RC(\infty)$ and let $x$ be the rigging of $\nu^{(a)} = (1^k)$.
To see $f_a f_a^*(\nu, J) = f_a^*f_a(\nu, J)$, begin by noting that $f_a$ and $f_a^*$ always add a new row when they act by Lemma~\ref{lemma:imaginary_partition}.
Therefore, we have that $f_a^* f_a(\nu, J)$ adds rows with riggings
\[
x = -\frac{A_{aa}}2
\qquad\qquad
x^* = p_1^{(a)} - \frac{3A_{aa}}2
\]
by $f_a$ and $f_a^*$ respectively and changes all other riggings by $-A_{aa}$. Similarly, $f_a f_a^*(\nu, J)$ adds rows with riggings $x$ and $x^*$, but in the oppose order, and changes all other riggings by $-A_{aa}$. Hence, we have $f_a f_a^*(\nu, J) = f_a^*f_a(\nu, J)$.

Next, note that for any $(\widetilde{\nu}, \widetilde{J}) \in \RC(\infty)$, we have $\varphi_a(\widetilde{\nu}, \widetilde{J}) = 0$ if and only if $\widetilde{\nu}^{(b)} = \emptyset$ whenever $a =b$ or $A_{ab} \neq 0$. Thus, from the definition of $f_a$ and $f_a^*$, we have that the following are equivalent:
\begin{itemize}
\item $\kappa_a(\nu, J) = 0$;
\item $p_1^{(a)} = -k A_{aa}$, where $\nu^{(a)} = (1^k)$;
\item the riggings of $(\nu, J)^{(a)}$ are $\{-(2m-1)A_{aa}/2 \mid 1 \leq m \leq k \}$.
\end{itemize}
Therefore, assume $\kappa_a(\nu, J) = 0$. Then $f_a^*$ adds a row with a rigging of $-(2k+1)A_{aa}/2$ and $f_a$ adds a row with a rigging of $-A_{aa}/2$ and changes all of the other riggings to 
\[
\frac{(-2m-1) A_{aa}}{2} - A_{aa} = \frac{-\bigl(2(m+1)-1\bigr) A_{aa}}{2}. 
\]
Hence, we have $f_a(\nu, J) = f_a^*(\nu, J)$.
Now assume $\kappa_a(\nu, J) > 0$, which is equivalent to $p_1^{(a)} > -k A_{aa}$. By the previous analysis, $e_a^*$ removes the rows with the largest riggings, but it only does so if the corigging is $-A_{aa}/2$. However, the largest corigging in $f_a(\nu, J)$ is
\[
p_1^{(a)} - \frac{A_{aa}}2 > -(k+1) A_{aa} - \frac{A_{aa}}2 = -\left(k + \frac12\right) A_{aa},
\]
and hence if $e_a^*$ removed all other rows, we would have a final corigging of strictly greater than $-A_{aa} / 2$.
Moreover, all other coriggings remain unchanged, and hence, we have $\tep_a^*\bigl( f_a(\nu, J) \bigr) = \tep_a^*(\nu, J)$.

Furthermore, it is clear that $\kappa_a(\nu, J) = 0$ if and only if $\kappa_a^*(\nu, J) = 0$. Therefore, the claim follows by Theorem~\ref{thm:Binf_recog}.
\end{proof}

Therefore, by Definition~\ref{def:RC_crystal_ops} and Definition~\ref{def:RC_star_crystal_ops}, we have the following.

\begin{cor}
The $*$-involution on $\RC(\infty)$ is given by replacing every rigging $x$ of a row of length $i$ in $(\nu, J)^{(a)}$ by the corresponding corigging $p_i^{(a)} - x$ for all $(a, i) \in \HH$.
\end{cor}

\section{Characterization in the purely imaginary case}\label{sec:purely_imaginary}

In this section, we give an explicit characterization of the rigged configurations in the purely imaginary case (\textit{i.e.}, when $I^{\im} = I$).

\begin{ex}
\label{ex:purely_imaginary}
Let $I = \{1,2\}$ and
\[
A = \begin{pmatrix} -2\alpha & -\beta \\ -\gamma & -2\delta \end{pmatrix},
\]
such that $\alpha,\beta,\gamma,\delta \in \ZZ_{\ge0}$ (so $I = I^\im$).  The top part of the crystal graph $\RC(\infty)$ is pictured in Figure \ref{fig:pureim}.  Set 
\[
(\nu,J) = f_1^3f_2(\nu_\emptyset,J_\emptyset) = \ \
\begin{tikzpicture}[scale=.4,baseline=-30]
    \rpp{1,1,1}{5\alpha,3\alpha,\alpha}{6\alpha+\beta,6\alpha+\beta,6\alpha+\beta}
   \begin{scope}[xshift=7cm]
    \rpp{1}{\delta+3\gamma}{2\delta+3\gamma}
   \end{scope}
   \end{tikzpicture}\ .
\]
Then
\[
f_2(\nu,J) = \ \
\begin{tikzpicture}[scale=.4,baseline=-30]
    \rpp{1,1,1}{5\alpha+\beta,3\alpha+\beta,\alpha+\beta}{6\alpha+2\beta,6\alpha+2\beta,6\alpha+2\beta}
   \begin{scope}[xshift=8cm]
    \rpp{1,1}{3\delta+3\gamma,\delta}{4\delta+3\gamma,4\delta+3\gamma}
   \end{scope}
   \end{tikzpicture}
\]
and
\[
f_2^*(\nu,J) = \ \
\begin{tikzpicture}[scale=.4,baseline=-30]
    \rpp{1,1,1}{5\alpha,3\alpha,\alpha}{6\alpha+2\beta,6\alpha+2\beta,6\alpha+2\beta}
   \begin{scope}[xshift=7cm]
    \rpp{1,1}{3\delta+3\gamma,\delta+3\gamma}{4\delta+3\gamma,4\delta+3\gamma}
   \end{scope}
   \end{tikzpicture}\ .
\]
More generally, if we consider the generic element
\[
f_1^{j_1} f_2^{k_1} \cdots f_1^{j_z} f_2^{k_z} (\nu_{\emptyset}, J_{\emptyset}) \in \RC(\infty),
\]
where $j_q, k_q > 0$ except possibly $j_1 = 0$,
then we have $\nu^{(1)} = (1^{j_1 + \cdots + j_z})$ and $\nu^{(2)} = (1^{k_1 + \cdots + k_z})$ with
\begin{align*}
J_1^{(1)} & = \{ (2j_1 + \cdots + 2j_z - 1)\alpha + (k_1 + \cdots + k_{z-1})\beta,
\\ & \hspace{20pt} \dotsc, (2j_1 + \cdots + 2j_{z-1} + 1)\alpha + (k_1 + \cdots + k_{z-1})\beta,
\\ & \hspace{20pt} \dotsc,
\\ & \hspace{20pt} (2j_1 + 2j_2 - 1)\alpha + k_1\beta, \dotsc, (2j_1 + 1)\alpha + k_1\beta,
\\ & \hspace{20pt} (2j_1 - 1)\alpha, \dotsc, \alpha\},
\allowdisplaybreaks \\
J_1^{(2)} & = \{ (2k_1 + \cdots + 2k_z - 1)\delta + (j_1 + \cdots + j_z)\gamma,
\\ & \hspace{20pt} \dotsc, (2k_1 + \cdots + 2k_{z-1} + 1)\delta + (j_1 + \cdots + j_z) \gamma,
\\ & \hspace{20pt} \dotsc,
\\ & \hspace{20pt} (2k_1 + 2k_2 - 1)\delta + (j_1 + j_2)\gamma, \dotsc, (2k_1 + 1)\delta + (j_1 + j_2) \gamma,
\\ & \hspace{20pt} (2k_1 - 1)\delta + j_1\gamma,\dotsc, \delta + j_1 \gamma\}.
\end{align*}
Note that since $\beta, \gamma > 0$, given such a $J_1^{(1)}$ and $J_1^{(2)}$, it is easy to see that we can uniquely solve for $j_1, \dotsc, j_z$ and $k_1, \dotsc, k_z$.
\end{ex}

\begin{figure}[t]
\include{pure_im}
\caption{Top of the crystal graph for a purely imaginary Borcherds--Cartan matrix in terms of rigged configurations using the Borcherds--Cartan matrix from Example~\ref{ex:purely_imaginary}.  Here, the blue arrows correspond to $f_1$ and the red arrows correspond to $f_2$.}
\label{fig:pureim}
\end{figure}

Let $A = (A_{ab})$ be a purely imaginary Borcherds--Cartan matrix.  Let $(\nu,J)$ be a rigged configuration such that $\nu = \bigl((1^{k_a}) \mid a \in I\bigr)$.  Given such a rigged configuration, write $J_1^{(a)} = \{x^{(a)}_1, \dotsc, x^{(a)}_{k_a}\}$ for each $a\in I$. We assume $x^{(a)}_1 \geq \cdots \geq x^{(a)}_{k_a}$.
We say $\{x^{(a)}_j \geq x^{(a)}_{j+1} \geq \cdots \geq x^{(a)}_{j'}\}$ is an \defn{$a$-string} if
\begin{itemize}
\item $x^{(a)}_q - x^{(a)}_{q+1} = -A_{aa}$ for all $j \leq q < j'$,
\item $x^{(a)}_{j-1} - x^{(a)}_j \neq -A_{aa}$, and
\item $x^{(a)}_{j'} - x^{(a)}_{j'+1} \neq -A_{aa}$.
\end{itemize}
Note that this agrees with a $a$-string of crystal operators if $x^{(a)}_{j'} = -\frac{1}{2} A_{aa}$.
This can be seen in the generic element at the end of Example~\ref{ex:purely_imaginary}.

\begin{ex}
Let $A = (A_{ab})_{a,b\in I}$ be a Borcherds--Cartan matrix with $I = I^\im = \{1,2,3\}$.  Then 
\[
f_2^3 f_3^2 f_1^2 f_3(\nu_\emptyset,J_\emptyset) = \ 
\begin{tikzpicture}[scale=.4,baseline=-22]
    \rpp{1,1}{-\frac{3}{2}A_{11}-3A_{12}-3A_{13}, -\frac{1}{2}A_{11}-3A_{12}-3A_{13}}{}
   \begin{scope}[xshift=11cm]
    \rpp{1,1,1}{-\frac{5}{2}A_{22},-\frac{3}{2}A_{22},-\frac{1}{2}A_{22}}{}
   \end{scope}
   \begin{scope}[xshift=16cm]
    \rpp{1,1,1}{-\frac{5}{2}A_{33}-A_{31}-3A_{32},-\frac{3}{2}A_{33}-3A_{32},-\frac{1}{2}A_{33}-3A_{32}}{}
   \end{scope}
   \end{tikzpicture}.
\]
Thus, the resulting rigged configuration has a $1$-string of size $2$, a $2$-string of size $3$.
For the $3$-strings, if $A_{31} = 0$, then there is a single $3$-string of size $3$, otherwise there are two $3$-strings of sizes $2$ and $1$.
\end{ex}

\begin{dfn}
We say $(\nu, J)$ is \defn{balanced} if $\nu^{(a)} = (1^{k_a})$ for all $a \in I$ and there exists a total ordering $(\Sigma_1, \dotsc, \Sigma_m)$ on $\{ S^{(a)}_j \mid a \in I,\ 1 \leq j \leq q_a \}$, where $\{S^{(a)}_1, \dotsc, S^{(a)}_{q_a}\}$ is the decomposition of $\nu^{(a)}$ into $a$-strings, such that
\begin{equation}
\label{eq:balanced_condition}
\overline{\Sigma}_j = -\frac{1}{2} A_{aa} - \sum_{k=1}^{j-1} A_{aa'} \left\lvert \Sigma_k \right\rvert,
\end{equation}
where $\Sigma_j = S^{(a)}_q$ and $\Sigma_k = S^{(a')}_{q'}$ with $\overline{\Sigma}_j$ denoting the smallest rigging of the $a$-string $\Sigma_j$.
Note that we vacuously have $(\nu_{\emptyset}, J_{\emptyset})$ being balanced.
\end{dfn}

\begin{prop}
Let $A$ be a purely imaginary Borcherds--Cartan matrix. The set of balanced rigged configurations equals $\RC(\infty)$.
\end{prop}

\begin{proof}
We need to show that the set of balanced rigged configurations is closed under the crystal operators and connected to $(\nu_{\emptyset}, J_{\emptyset})$.
From Lemma~\ref{lemma:imaginary_partition}, we see that we must have $\nu^{(a)} = (1^{k_a})$.
From Equation~\eqref{eq:balanced_condition} with $j = 1$ and the crystal operators, the only highest weight balanced rigged configuration is $(\nu_{\emptyset}, J_{\emptyset})$.
Therefore, it is sufficient to show that for a balanced rigged configuration $(\nu, J)$, the rigged configuration $e_a(\nu, J) = (\widetilde{\nu}, \widetilde{J}) \in \RC(\infty)$ is also balanced.
However, it is straightforward to see this from the definition of the crystal operators, which removes the smallest rigging from $\Sigma_1$, and Equation~\eqref{eq:balanced_condition}.
\end{proof}

We finish this section with an aside about the crystal operators in the purely imaginary case. We remark that the following fact is implicitly why $\RC(\infty)$ is described by balanced rigged configurations.

\begin{prop}
Let $a, a' \in I^{\im}$.
If $A_{aa'} = 0$, then the crystal operators $f_a$ and $f_{a'}$ commute. Otherwise, $f_a$ and $f_{a'}$ are free.
\end{prop}

\begin{proof}
It is sufficient to consider the rank-$2$ case with $I = I^{\im} = \{1,2\}$ with
\[
A = \begin{pmatrix} -2\alpha & -\beta \\ -\gamma & -2\delta \end{pmatrix}.
\]
If $\beta = \gamma = 0$, then it is clear that the crystal operators commute. If $\beta, \gamma > 0$, then consider a rigged configuration $(\nu, J) \in \RC(\infty)$ such that without loss of generality $e_1(\nu, J) \neq \zero$.
Hence, we have $\min J_1^{(1)} = \alpha$. If $f_2^k(\nu, J) = (\widetilde{\nu}, \widetilde{J})$ for some $k \in \ZZ_{>0}$, then we have $\min \widetilde{J}_1^{(1)} = \alpha + k\beta > \alpha$. Hence $e_1 (\widetilde{\nu}, \widetilde{J}) = \zero$, and the claim follows.
\end{proof}

As a consequence, in the purely imaginary case, the elements of $B(\infty)$ are in bijection with a right-angled Artin monoid:
$
\langle f_a \ |\ f_a f_{a'} = f_{a'} f_a \text{ if } A_{aa'} = 0 \rangle.
$
In particular, the Cayley graph of this monoid is isomorphic to the crystal graph.

\section{Highest weight crystals}
\label{sec:highest_weight}

We can describe highest weight crystals $B(\lambda)$ by utilizing Theorem~\ref{thm:cutout}. Fix some $\lambda \in P^+$. We describe the crystal $B(\lambda)$ using rigged configurations by defining new crystal operators $f'_a(\nu, J)$ as $f_a(\nu, J)$ unless $p_i^{(a)} + \inner{h_a}{\lambda} < x$ for some $(a,i) \in \HH$ and $x \in J_i^{(a)}$ or $\varphi_a(\nu, J) = 0$ for $a \in I^{\im}$, in which case $f'_a(\nu, J) = 0$. Let $\RC(\lambda)$ denote the closure of $(\nu_{\emptyset}, J_{\emptyset})$ under $f'_a$.

\begin{thm}
\label{thm:highest_weight_RC}
Let $\lambda \in P^+$. Then $\RC(\lambda) \iso B(\lambda)$.
\end{thm}

The proof of Theorem~\ref{thm:highest_weight_RC} is the same as~\cite[Thm.~4]{SalS16II}.

Next, we can characterize the image of $B(\lambda)$ inside $B(\infty)$ using the $\ast$-involution in analogy to~\cite[Prop.~8.2]{K95}. Recall the crystal $T_{\lambda}$ from Example~\ref{ex:T_crystal}.

\begin{cor}
\label{cor:classification_hw_star}
Let $\lambda \in P^+$. Then we have
\[
\RC(\lambda) \iso \left\{
(\nu,J) \otimes t_{\lambda} \in \RC(\infty) \otimes T_{\lambda} \mid
\begin{gathered}
  \varepsilon_a^*(\nu, J) \leq \inner{h_a}{\lambda} \text{ for all } a \in I^{\re}, \\
  e_a^*(\nu, J) = \zero \text{ if } \inner{h_a}{\lambda} = 0 \text{ for all } a \in I^{\im}
\end{gathered}
\right\}.
\]
\end{cor}

\begin{proof}
For $a \in I^{\re}$, this was done in~\cite[Prop.~5]{SalS16II}.

For $a \in I^{\im}$, consider some $(\nu, J) \in \RC(\infty)$, and it is sufficient to consider the case when $\inner{h_a}{\lambda} = 0$. If $\inner{h_a}{\wt(\nu, J)} = p_{\infty}^{(a)} = 0$, then we have $\nu^{(a)} = \emptyset$ and $p_1^{(a)} = 0$ since $-A_{aa'} \geq 0$ for all $a' \in I$.
Therefore, the smallest nonnegative corigging in $\bigl( f_a(\nu, J) \bigr)^{(a)}$ is $-A_{aa}/2$.
Let $(\nu', J') = f_{\vec{a}} f_a (\nu, J)$ be a rigged configuration obtained from $f_a(\nu, J)$ after applying some (possibly empty) sequence $f_{\vec{a}}$ of crystal operators. Since the crystal operators preserve coriggings and $f_a$ will never again act on this row, the smallest nonnegative corigging of $(\nu', J')^{(a)}$ is $-A_{aa}/2$. Hence, we have $e_a^*(\nu', J') \neq 0$. Similarly, if $\inner{h_a}{\wt(\nu, J)} > 0$, then the smallest corigging of $\bigl( f_a(\nu, J) \bigr)^{(a)}$ is strictly larger than $-A_{aa}/2$. Hence, we have $e_a^* f_{\vec{a}} f_a (\nu, J) = 0$ for any sequence of crystal operators $f_{\vec{a}}$.
\end{proof}

\begin{remark}
\label{rem:alt_proof_classification}
An alternative (abstract) proof of Corollary~\ref{cor:classification_hw_star} can be done using~\cite[Prop.~8.2]{K95} for $a \in I^{\re}$ and the recognition theorem (Theorem~\ref{thm:Binf_recog}) for $a \in I^{\im}$. In particular, for any $a \in I^{\im}$ and $v \in B(\infty)$ we have $\inner{h_a}{\wt(v)} = 0$ if and only if $\kappa_a(v) = 0$, $\inner{h_a}{\lambda} = 0$, and $e_a^* v = \zero$ (alternatively, $\widetilde{\varepsilon}_a^*(v) = 0$). Moreover, $e_a^*$ commutes with $f_{a'}$ for all $a' \in I$.
\end{remark}

\bibliography{borcherds_RC}{}
\bibliographystyle{amsplain}
\end{document}